\documentclass[a4paper,12pt,reqno,oneside]{amsart}

\usepackage{setspace} 
\usepackage{typearea} % European margins

\usepackage{amscd,amsmath,amssymb,verbatim}
\usepackage{xr-hyper}
\usepackage{cite}
\usepackage[colorlinks,backref,final,bookmarksnumbered,bookmarks]{hyperref}
\usepackage[backrefs]{amsrefs}
\usepackage{ifthen}
\newcommand{\arxiv}[2][]{\ifthenelse{\equal{#1}{}}
{\href{http://arxiv.org/abs/#2}{\tt arXiv:#2}}
{\href{http://arxiv.org/abs/math/#2}{\tt arXiv:math.#1/#2}}}

\makeatletter
\renewcommand\subsection{\@startsection
{subsection}{2}{0cm} % name, level, indent
{-\baselineskip}     % before-skip
{0.5\baselineskip}   % after-skip
{\sffamily}} % style
\makeatother

\theoremstyle{plain}
\newtheorem{theorem}{Theorem}[section]
\newtheorem{lemma}[theorem]{Lemma}
\newtheorem{corollary}[theorem]{Corollary}
\newtheorem{proposition}[theorem]{Proposition}
\newtheorem{problem}[theorem]{Problem}
\newtheorem{conjecture}[theorem]{Conjecture}

\newtheoremstyle{concise}
{}{}{}{}{\bfseries}{}{ }{\thmnumber{#2.}\thmnote{ #3.}}
\theoremstyle{concise}
\newtheorem{definition}[theorem]{}

\newtheoremstyle{remark}
{}{}{}{}{\itshape}{}{ }{\thmname{#1}\thmnumber{ \itshape #2.}}
\theoremstyle{remark}
\newtheorem{remark}[theorem]{Remark}

\def\Z{\mathbb{Z}}

\def\N{\mathbb{N}}
\def\but{\setminus}
\def\x{\times}
\def\eps{\varepsilon}
\def\phi{\varphi}
\def\i{\subset}
\def\emb{\hookrightarrow}
\def\imp{$\Rightarrow$}
\def\xr#1{\xrightarrow{#1}}
\def\dirlim{\lim\limits_\rightarrow{}}
\def\invlim{\lim\limits_\leftarrow{}}
\def\derlim{\lim\limits_\leftarrow{^1}{}}
\def\:{\colon}
\def\Cel{\lfloor} \def\Cer{\rfloor}
\def\Fll{\lceil} \def\Flr{\rceil}
\def\downscale#1{\mathchoice{\raisebox{1pt}{$\scriptstyle#1$}}
{\raisebox{1pt}{$\scriptstyle#1$}}{\raisebox{.5pt}{$\scriptscriptstyle#1$}}
{{\scriptscriptstyle#1}}}
\def\cel{\downscale\Cel} \def\cer{\downscale\Cer}
\def\fll{\downscale\Fll} \def\flr{\downscale\Flr}

\DeclareMathOperator{\st}{st}
\DeclareMathOperator{\lk}{lk}
\DeclareMathOperator{\Cl}{Cl}
\DeclareMathOperator{\Int}{Int}
\DeclareMathOperator{\Tel}{Tel}
\DeclareMathOperator{\Hom}{Hom}
\DeclareMathOperator{\im}{im}
\DeclareMathOperator{\rk}{rk}
\DeclareMathOperator{\id}{id}

\begin{document}
\title[Absolute retracts in products of dendrites]
{Contractible polyhedra in products of trees and absolute retracts
in products of dendrites}
\author{Sergey A. Melikhov and Justyna Zaj\c ac}
%\date{\today}
\address{Steklov Mathematical Institute of the Russian Academy
of Sciences, ul. Gubkina 8, Moscow 119991, Russia}
\address{Delaunay Laboratory of Discrete and Computational Geometry,
Yaroslavl' State University, 14 Sovetskaya st., Yaroslavl', 150000 Russia}
\email{melikhov@mi.ras.ru}
\address{Faculty of Mathematics, Informatics, and Mechanics,
University of Warsaw,
Banacha 2,
02-097 Warszawa,
Poland}
\email{jzajac@mimuw.edu.pl}
\thanks{The first author is supported by
Russian Foundation for Basic Research Grant No.\ 11-01-00822,
Russian Academy of Sciences program ``Mathematical methods of
construction and analysis of models of complex systems'',
Russian Government project 11.G34.31.0053 and
Federal Program ``Scientific and scientific-pedagogical staff of
innovative Russia 2009--2013''}

\begin{abstract}
We show that a compact $n$-polyhedron PL embeds in a product of $n$ trees if and only if it
collapses onto an $(n-1)$-polyhedron.
If the $n$-polyhedron is contractible and $n\ne 3$ (or $n=3$ and the Andrews--Curtis Conjecture holds), the product of trees may be assumed to collapse onto the image of the embedding.

In contrast, there exists a $2$-dimensional compact absolute retract $X$ such that $X\x I^k$
does not embed in any product of $2+k$ dendrites for each $k$.
\end{abstract}

\maketitle
%\tableofcontents
\section{Introduction}

All {\it spaces} shall be assumed to be metrizable.
By a {\it compactum} we mean a compact metrizable space.
A finite-dimensional compactum is an {\it ANR} if and only if it is locally
contractible; and an {\it absolute retract (AR)} if and only if it is a contractible ANR
(see \cite{Bo1}).
A one-dimensional compact AR is called a {\it dendrite}, and a
one-dimensional compact ANR is called a {\it local dendrite}.
An arbitrary connected one-dimensional compactum is sometimes called a {\it curve}.

\begin{theorem}[Nagata--Bowers \cite{Na}, \cite{Bow}; see also
\cite{St1}, \cite{St2}, \cite{BV}]
\label{thm:bowers}
Every $n$-dimensional compactum $X$ embeds in $D^n\x I$, where $D$ is a certain
dendrite.
\end{theorem}

It is well-known that every dendrite embeds in the $2$-cube $I^2$; thus
Theorem \ref{thm:bowers} may be viewed as an improvement of the classical
Menger--N\"obeling--Pontriagin theorem that every $n$-dimensional compactum
embeds in the $(2n+1)$-cube $I^{2n+1}$.

Theorem \ref{thm:bowers} is trivial in the case where $X$ is a polyhedron:

\begin{theorem}\label{bowers}
Every compact $n$-polyhedron embeds in a product of $n$ trees and $I$.
\end{theorem}

\begin{proof}
Given a triangulation $K$ of the given polyhedron $X$, let $S_i$ be the set of
all vertices of the barycentric subdivision $K'$ that are barycenters of
$i$-simplices of $K$.
The simplicial map $K'\to S_0*\dots*S_n$ is clearly an embedding.
Hence $X$ embeds in $I*S$, where $S=S_1*\dots*S_n$.
Next, $I*S=pt*CS$ is homeomorphic to $pt*(CS\cup S\x I)=I\x CS$.
Finally, the cone $CS$ is homeomorphic to the product of $n$ trees
$CS_1\x\dots\x CS_n$.
\end{proof}

The above argument yields an explicit embedding of every compact $n$-polyhedron
in $I^{2n+1}$, which we have not seen in the literature.
This is strange, for a part of this construction is certainly well-known; it yields

\begin{proposition}[\cite{GMR}, \cite{KKS}]\label{cone} The cone over every
compact $n$-polyhedron embeds in a product of $n+1$ trees.
\end{proposition}

\begin{theorem}[Borsuk--Patkowska \cite{Bo2}]\label{Patkowska}
The $n$-sphere $S^n$ does not embed in any product of $n$ dendrites,
for each $n\ge 0$.
\end{theorem}

Another proof of Theorem \ref{Patkowska} is given by the easy part of our Theorem \ref{characterization} below.

\begin{theorem}[Gillman--Matveev--Rolfsen \cite{GMR}]\label{GMR}
Every compact connected PL $n$-manifold with nonempty boundary embeds in a product of $n$ trees.
\end{theorem}

This was originally a consequence of Proposition \ref{cone} along with a ``reconstruction theorem'' announced in \cite{GMR}.
Another proof of Theorem \ref{GMR} is given by our Theorem \ref{characterization}
below, albeit the trees that it produces need not be cones over finite sets.

Nagata's original motivation for considering embedding into products of $1$-dimensional spaces
related to dimension theory (see \cite{Na}).
Borsuk's proof of the $2$-dimensional case of Theorem \ref{Patkowska} was a solution
to Nagata's problem; on the other hand, the author learned from W. Kuperberg,
a student of Borsuk who has generalized Theorem \ref{Patkowska} \cite{Ku}, that
Borsuk saw this result as a part of his work on the problem of uniqueness of
decomposition of ANRs into products.
Yet another motivation for embedding into products of trees was the Poincar\'e
Conjecture (now also known as Perelman's Theorem):

\begin{theorem}\label{Gillman}
(a) {\rm (Gillman \cite{Gi})} If a compact acyclic $3$-manifold embeds in the product of
a tree and $I^2$, then it is collapsible.

(b) {\rm (Zhongmou \cite{Zh})} Every compact connected $3$-manifold with nonempty boundary embeds in
the product of two triods and $I$.
\end{theorem}

A discussion of further results in the theory of embeddings into products of
dendrites (or curves) can be found in the recent paper \cite{KKS}, which itself
is a significant addition to this theory (see also additional details in
\cite{KKS$'$}).
We should mention

\begin{theorem}[Koyama--Krasinkiewicz--Spie\.z \cite{KKS}]
There exists a $2$-polyhedron that collapses onto a product of two graphs
but does not embed in any product of two graphs.
Yet it embeds in a product of two curves.
\end{theorem}

The $2$-polyhedron in question is $\Theta\x\Theta\underset{J=I\x\{0\}}{\cup}I\x I$,
where $\Theta$ is the suspension over the three-point set, and the arc $J$ lies
in the interior of a $2$-cell of $\Theta\x\Theta$ apart from one endpoint, which
lies in a ``corner'' of that $2$-cell.

\subsection{Embedding contractible polyhedra in products of trees}

\begin{theorem}\label{collapsible} Every collapsible compact $n$-polyhedron PL embeds
in a product of $n$ trees.
Moreover, the product of trees collapses onto the image of the embedding.
\end{theorem}

The embeddability in the $2$-dimensional case is due to Koyama, Krasinkiewicz and
Spie\.z \cite{KKS}.
The principal additional ingredient in our proof of the general case is
the Fisk--Izmestiev--Witte lemma  \cite[Lemma 57]{Fi}, \cite{Iz}, \cite{Wi}
(see also \cite[Lemma 3.1]{HW}, \cite{Ca}), which asserts that for every
finite set $C$ (the `palette') of cardinality $\#C\ge d+1$, every $C$-colored
combinatorial $(d-1)$-sphere is color-preserving isomorphic to the boundary of
a $C$-colored combinatorial ball.
(A simplicial complex is {\it $C$-colored} if its vertices are colored
by the elements of $C$ so that no edge connects two vertices of the same color.)

In particular, this lemma implies that if a triangulation of $S^2$ admits
a $4$-coloring, then it extends to a triangulation of the $3$-ball where
the link of every interior edge is (combinatorially) an even-sided polygon.
As observed by R. D. Edwards and others in 1970s, the converse to this also
holds: every such triangulation of the $3$-ball has a $4$-colorable
boundary (see references in \cite{Iz}).

The $2$-dimensional case of Theorem \ref{collapsible} involves only the trivial
case $d\le 1$ of the Fisk--Izmestiev--Witte lemma.

\begin{corollary}\label{characterization}
Let $P$ be a compact $n$-polyhedron.
The following are equivalent:

(i) $P$ PL embeds in a product of $n$ trees;

(ii) $P$ PL embeds in a product of an $(n-1)$-polyhedron and a tree;

(iii) $P$ collapses onto an $(n-1)$-polyhedron;

(iv) $P$ PL embeds in a collapsible compact $n$-polyhedron.
\end{corollary}

Here (iv)\imp (i) follows from Theorem \ref{collapsible}, (i)\imp (ii) is obvious, (ii)\imp (iii)
is easy (see below), and to see that (iii)\imp (iv) it suffices to note that if $P$ collapses
onto $Q$ then the amalgamated union $P\cup_Q CQ$ is collapsible, where $CQ$ is the cone over $Q$.

Alternatively, (i)\imp (iv) is obvious, and another proof of (iii)\imp (i) is given in \S2.

\begin{proof}[Proof of (ii)\imp (iii)]
Suppose that $P$ is embedded in $R\x T$, where $R$ is an $(n-1)$-polyhedron and $T$ is a tree,
and $P$ does not collapse onto any $(n-1)$-polyhedron.
Let $P_0$ be a triangulation of $P$.
Then $P_0$ collapses onto a (generally non-unique) simplicial complex $Q_0$ that does not collapse
onto any proper subcomplex.
Then $Q_0$ has no free faces, and it follows that $Q:=|Q_0|$ does not collapse onto any proper
subpolyhedron.
By the hypothesis $Q$ is of dimension precisely $n$.
The projection $f\:Q\subset R\x T\to R$ can be triangulated by a simplicial map $Q_1\to R_1$.
Let $p$ is a point in the interior $U$ of a top-dimensional simplex of $R_1$ such that
the corresponding fiber $F:=f^{-1}(p)$ is of dimension precisely $1$.
The projection $F\subset R\x T\to T$ is an embedding, so $F$ is a forest.
Thus $F$ collapses onto a finite set, but is not a finite set itself; so it must have
a free vertex.
On the other hand, $f^{-1}(U)\cong F\x U$ by Pontryagin's lemma
\cite[Proposition C]{Po}, \cite[Theorem 1.3.1]{Wil} (see also \cite[\S5]{Co1}).
Hence $Q_1$ has a free face.
Thus $Q$ collapses onto a proper subpolyhedron, which is a contradiction.
\end{proof}

\begin{corollary}[Koyama--Krasinkiewicz--Spie\.z \cite{KKS}] \label{2d-case}
An acyclic compact $2$-polyhedron $P$ embeds in a product of two trees if and only if $P$ is
collapsible.
\end{corollary}

\begin{remark}\label{contr-non-coll}
Let $P$ be a compact polyhedron with $H^1(P)=0$. If $P$ embeds in a product of $n$ graphs
then it embeds in a product of $n$ trees, namely in the product of (appropriate compact
subtrees of) the universal covers of the $n$ graphs.
Thus ``trees'' can be replaced with ``graphs'' in Corollary \ref{2d-case}
in accordance with \cite{KKS}.
(In fact, it was shown in \cite{KKS} that an acyclic non-collapsible compact
$2$-polyhedron does not embed in any product of two curves.)
\end{remark}

\begin{corollary}\label{contractible}
Let $P$ be an $n$-polyhedron.
For $n\ne 3$, the following are equivalent:

(i) some product of $n$ trees collapses onto a PL copy of $P$;

(ii) $P$ collapses onto a contractible $(n-1)$-polyhedron;

(iii) some collapsible compact $n$-polyhedron collapses onto a PL copy of $P$.

\noindent
For $n=3$, the same holds with ``contractible'' replaced by ``$3$-deformable to a point''.
\end{corollary}

A polyhedron $P$ is said to be {\it $n$-deformable} to a polyhedron $Q$
if they are related by a sequence of collapses and expansions (i.e.\ the inverses
of collapses) through polyhedra of dimensions $\le n$.
The Andrews--Curtis Conjecture asserts that all contractible $2$-polyhedra
$3$-deform to a point (see \cite{2DH}, \cite{Mat}).
Among its motivations (cf.\ Curtis \cite[\S2]{Cu}) we mention that it would imply%
\footnote{By general position every $2$-polyhedron $P$ immerses in $I^4$, and
therefore embeds in a $4$-manifold $M$.
Let $N$ be a regular neighborhood of $P$ in $M$.
If $P$ $3$-deforms to a point, then the double of $N$ is the $4$-sphere
(see \cite[Assertion (59) in Ch.\ I]{2DH}).}
that every contractible $2$-polyhedron PL embeds in $I^4$.

\begin{proof} (iii)\imp(i) follows from Theorem \ref{collapsible} and
(i)\imp(ii) follows from Corollary \ref{characterization}. To prove (ii)\imp
(iii), suppose that $P$ collapses onto an $(n-1)$-polyhedron $Q$, and either
$Q$ is contractible, or $n=3$ and $Q$ $3$-deforms to a point.
Then by a result of Kreher--Metzler and Wall, there exists an
$(n-1)$-polyhedron $R$ such that $R$ collapses onto a PL copy of $Q$ and
$R\x I$ is collapsible \cite[Satz 1a, Satz 1]{KMe} (see also
\cite[\S XI.4]{2DH} for an outline of Kreher and Metzler's proof in English).
Let $S$ be the amalgamated union $P\cup_{Q=Q\x\{0\}} R\x I$.
Then $S\searrow R\x I\searrow pt$ and $S\searrow P\cup_Q R\searrow P$.
\end{proof}

\begin{remark}
For each $n\ge 3$ it is easy to construct a non-collapsible $n$-polyhedron
that collapses onto a contractible $(n-1)$-polyhedron (e.g.\ $I^n\vee\,$cone$(f)$
will do, where $f$ is any degree $0$ PL surjection $S^{n-2}\to S^{n-2}$).
A more interesting example is due to M. M. Cohen, who constructed for each $n\ge 4$
a contractible $(n-1)$-polyhedron $Q$ such that $Q\x I$ is not collapsible
\cite{Co2}.
Other constructions (with very different proofs) are now known: $P\x I^{k-2}$
is not collapsible if $P$ is the suspension over a $(k-1)$-dimensional spine of
a non-simply-connected homology $k$-sphere \cite{BCC}, and $P\x I^q$ is
not collapsible if $P$ is a certain ``$(3q+6)$-dimensional dunce hat''
\cite{BM}.
\end{remark}

A {\it free deformation retraction} of a space $X$ onto a subspace $Y$
is a homotopy $h_t\:X\to X$ starting with $h_0=\id$, ending with a retraction
$h_1$ of $X$ onto $Y$, and such that $h_t h_s=h_{\max(s,t)}$ for all $s,t\in [0,1]$.
A space is {\it freely contractible} if it freely deformation retracts onto
a point.
Collapsibility is known to be strictly stronger than topological collapsibility
\cite{BCC}, \cite{BM} and consequently than free contractibility; however,
in the case of $2$-polyhedra the three notions are equivalent \cite{Is}.

\begin{conjecture} A compact $n$-polyhedron collapses onto an $(n-1)$-polyhedron
if and only if it freely deformation retracts onto an $(n-1)$-polyhedron.
\end{conjecture}

\begin{remark} The proof of Theorem \ref{collapsible} involves
a (non-straightforward) construction of a collapsible cubulation of the given
collapsible polyhedron, which might be of interest in its own right.
Another such construction (a more straightforward one) has been used to characterize
collapsible polyhedra in the language of abstract convexity theory \cite{vdV},
and to establish the `only if' part of Isbell's conjecture: a compact polyhedron
is collapsible if and only if it is injectively metrizable \cite{MT},
\cite[Chapter VI]{Ve}.
(Isbell himself proved that the two conditions are equivalent for $2$-polyhedra
\cite{Is}.)
\end{remark}

\subsection{Embedding absolute retracts in products of dendrites}
A map $f\colon X\to Y$ is called an {\it $\eps$-map} with respect to some metric
on $X$ if every its point-inverse $f^{-1}(pt)$ is of diameter at most $\eps$.
A compactum $X$ is said to {\it quasi-embed} in a space $Y$ if for some
(or equivalently, every) metric on $X$, it admits an $\eps$-map into $Y$ for
each $\eps>0$.
We refer to \cite{SSS} for a definitive discussion of the (quite subtle)
difference between embeddability and quasi-embeddability of compact polyhedra
in $I^m$.

Our paper was originally motivated by the following problem.

\begin{problem}[Koyama, Krasinkiewicz, Spie\.z \cite{KKS$'$}]\label{prob:kks}
Suppose that $X$ is a compactum, quasi-embeddable in the $n$th power of
the Menger curve.
Can $X$ be embedded there?
\end{problem}

This problem appears as Problem 1.4 in \cite{KKS$'$} with the following
comments: ``Our next problem is of great interest, we believe it has
affirmative solution.''

In the present paper, we shall prove

\begin{theorem} \label{thm:main}
There exists a $2$-dimensional compact AR $X$ such that $X\x I^k$ quasi-embeds
in a product of $2+k$ dendrites but does not embed in any product of $2+k$ curves,
for each $k\ge 0$.
\end{theorem}

The proof of the higher-dimensional (i.e.\ $k\ge 1$) case is similar to (and only
three lines longer than) the proof of the two-dimensional case.
Similar arguments show that the Cartesian power $X^k$ quasi-embeds in a product
of $2k$ dendrites, but does not embed in any product of $2k$ curves.

\begin{remark}
A few months after we shared our proof of the two-dimensional case of
Theorem \ref{thm:main} with J. Krasinkiewicz and S. Spie\.z, they found their
own solution of Problem \ref{prob:kks} \cite{KS}.
Compared to ours, it is amazingly simple (modulo their previous work with
A. Koyama) --- at least when slightly modified as follows.

The dunce hat $D$ \cite{Ze} (also known as the Borsuk tube \cite{Bor}, \cite{KS})
is easily seen to be the quotient of a collapsible polyhedron $\hat D$ by its only
free edge.
Indeed, the link $L$ of the $0$-cell $e_0$ of $D$ is homeomorphic to
$S^1\x\partial I\cup pt\x I$, where $pt\in S^1$ (cf.\ \cite[Fig.\ 5]{Ze}).
Let $\pi\:L\to I$ be the projection.
The star $S$ of $e_0$ in some triangulation of $D$ is homeomorphic to the cone over
$L$, which can be viewed as the mapping cylinder of the constant map $L\to pt$;
we define $\hat D$ by replacing $S$ with the mapping cylinder $MC(\pi)$.
The target space $I$ of $\pi$ is identified with a free edge $J$ in $\hat D$, and
clearly $\hat D$ is collapsible.

The quotient map $\hat D\to \hat D/J=D$, being cell-like, is an $\eps$-homotopy
equivalence for each $\eps>0$ by Chernavsky's lemma \cite[Lemma 1]{KC}; in particular,
for each $\eps>0$ there exists an $\eps$-map $f_\eps\:D\to\hat D$.
(Specifically, $f_\eps$ is the identity outside $S$, and $f_\eps|_S$ is the
composition $S\xr{h}L\x I\underset{L\x\{1\}=L*\emptyset}\cup L*pt\xr{g}MC(\pi)$,
where $h$ is a homeomorphism such that $h^{-1}(L*pt)$ lies in the
$\frac\eps2$-neighborhood of $e_0$, and $g$ combines the quotient map
$L\x I\to MC(\pi)$ with a null-homotopy $L*pt\to I$ of $\pi$.)
Since $\hat D$ is collapsible, it embeds in a product of two trees (\cite{KKS};
see Corollary \ref{2d-case} above), so $D$ quasi-embeds there; on the other
hand, $D$ does not embed in any product of two curves since it is contractible
but not collapsible (\cite{KKS}; see Remark \ref{contr-non-coll} above).

As observed in \cite{KS}, similar arguments show that the Cartesian power $D^k$
quasi-embeds in a product of $2k$ trees, but does not embed in any product of $2k$
curves.
(This uses the more general result of \cite{KKS} that no polyhedron $P$ with
$\rk H^1(P)<n$ and $H^n(P,P\but\{x\})\ne 0$ for each $x\in P$ embeds in a product
of $n$ curves.)
\end{remark}

\begin{remark} Zeeman showed that $D\x I$ is collapsible \cite{Ze}, where $D$ is
the dunce hat.
Hence $D\x I$ embeds in a product of $3$ trees by Theorem \ref{collapsible}.
So the absolute retract $X$ in Theorem \ref{thm:main} cannot be replaced by $D$.
Moreover, it cannot be replaced by {\it any} $2$-polyhedron $R$, since $R\x I$
embeds in a product of $3$ trees by Proposition \ref{cone}.
\end{remark}

\begin{conjecture}\label{conj} (a) If a compact $n$-polyhedron $P$ quasi-embeds in
a product of $n$ dendrites, then $P\x I$ embeds in a product of $n+1$ trees.

(b) Same if $P$ is a co-locally contractible (see \S\ref{co-local}) $n$-dimensional
compactum.
\end{conjecture}

Theorem \ref{thm:main} should be compared with the following results.

\begin{theorem}[Melikhov--Shchepin \cite{MS}]
(a) If $X$ is a compact $n$-dimensional ANR that quasi-embeds in $I^{2n-1}$, $n>3$,
then $X\x I$ embeds in $I^{2n}$.

(b) If $X$ is an acyclic $n$-dimensional compactum, $m>\frac{3(n+1)}2$ and $k>0$,
then the following are equivalent: (i) $X$ embeds in $I^m$; (ii) $X\x I^k$
embeds in $I^{m+k}$; (iii) $X\x T^k$ embeds in $I^{m+2k}$, where $T$
denotes the triod.
\end{theorem}

In conclusion we note that the proof of non-embeddability in Theorem \ref{thm:main} involves the same kinds of local geometry and local algebra as that in the following

\begin{theorem}[Melikhov--Shchepin \cite{MS}] For each $n>1$ there exists
a compact $n$-dimensional ANR, quasi-embeddable but not embeddable in $I^{2n}$.
\end{theorem}

\section{Collapsible polyhedra}

We use the following combinatorial notation \cite[Chapter 2]{M3}.
Given a poset $P$ and a $\sigma\in P$, the {\it cone} $\fll\sigma\flr$
is the subposet of all $\tau\in P$ such that $\tau\le\sigma$, and the
{\it dual cone} $\cel\sigma\cer$ is the subposet of all $\tau\in P$ such that
$\tau\ge\sigma$.
The {\it link}
$\lk(\sigma,P)$ is the subposet of all $\tau\in P$ such that
$\tau>\sigma$, and the {\it star} $\st(\sigma,P)$ is the subposet of all
$\rho\in P$ such that $\rho\le\tau$ for some $\tau\in\cel\sigma\cer$.
If $K$ is a simplicial complex (viewed as a poset of nonempty faces ordered by
inclusion), and $\sigma\in K$, then $\lk(\sigma,K)$ is a simplicial complex,
and $\st(\sigma,K)$ is isomorphic to $\fll\sigma\flr*\lk(\sigma,K)$.%
\footnote{Our $\lk(\sigma,P)$ is a standard notion of link in modern Topological
Combinatorics; we shall need it when $P$ is a cubical complex (where every cone
is isomorphic to the poset of nonempty faces of a cube).
The notion of link in Combinatorial Topology of 1960s was something slightly
different: being defined only when $P$ is a simplicial complex, it is canonically
{\it isomorphic} to our $\lk(\sigma,P)$ but is not {\it identical} with it.}

Here the join is defined as follows.
The {\it dual cone} $C^*P$ of the poset $P$ consists of $P$ together with
an additional element $\hat 0$ that is set to be less than every element of $P$.
The {\it coboundary} $\partial^*Q$ of a dual cone $Q=C^*P$, is the
original poset $P$.
(Note the relation with coboundary of cochains.)
The {\it product} $P\x Q$ of two posets consists of pairs $(p,q)$, where
$p\in P$ and $q\in Q$, ordered by $(p,q)\le (p',q')$ if $p\le q$ and
$p'\le q'$.
The {\it join} $P*Q=\partial^*(C^*P\x C^*Q)$.
Note that $P*Q=C^*P\x Q\cup P\x C^*Q$ (union along $P\x Q$).

The {\it canonical subdivision} $P^\#$ is the poset of all order intervals
of $P$, ordered by inclusion.
If $K$ is a simplicial complex, then $(C^*K)^\#$ is a cubical complex.
Conversely, if $Q$ is a cubical complex and $q\in Q$, then $\lk(q,Q)$ is
a simplicial complex, and $\st(q,Q)$ is isomorphic to
$\fll q\flr\x(C^*\lk(q,Q))^\#$.
Moreover, $\lk((p,q),P\x Q)$ is isomorphic to $\lk(p,P)*\lk(q,Q)$.
The details can be found in \cite[Chapter 2]{M3}.

Theorem \ref{collapsible} now follows from

\begin{lemma}\label{main lemma}
Let $K\searrow L$ be a simplicial collapse of simplicial complexes
and let $T_1,\dots,T_n$ be trees, so that $T=T_1\x\dots\x T_n$ is a cubical complex.
Suppose that $f\:|L|\to |T|$ is a PL embedding such that $f(|\sigma|)$ is cubulated
by a subcomplex of $T$ for every simplex $\sigma$ of $L$.
Then each $T_i$ embeds in a larger tree $\tilde T_i$ and $f$ extends to
a PL embedding $\bar f\:|K|\to|\tilde T|$, where
$\tilde T=\tilde T_1\x\dots\x\tilde T_i$, such that $\bar f(|\sigma|)$ is cubulated
by a subcomplex of $\tilde T$ for every simplex $\sigma$ of $K$.

Moreover, $|T|\cap\bar f(|K|)=f(|L|)$, and $|\tilde T|$ collapses onto $|T|\cup\bar f(|K|)$.
\end{lemma}

\begin{proof}
Arguing by induction, we may assume that $K\searrow L$ is an elementary
simplicial collapse.
Let $Q$ denote the subcomplex of $T$ cubulating $f(|L|)$, and let $B$ be
the subcomplex of $Q$ cubulating the image of the topological frontier
of $|L|$ in $|K|$.
We may now forget $K$, $L$ and $f$, remembering only that $|B|$ is a PL ball of
some dimension $k<n$.
We thus want to construct trees $\tilde T_i\supset T_i$ and a subcomplex $\beta$ of
$\tilde T_1\x\dots\x\tilde T_n$ such that $\beta\cap Q=B$ and $|\beta|$ is
a PL $(k+1)$-ball.

The boundary of $|B|$ is cubulated by a subcomplex $\partial B$ of $B$.
Given a face $q=q_1\x\dots\x q_n$ of $B\but\partial B$, we have
$\lk(q,T)\simeq\lk(q_1,T_1)*\dots*\lk(q_n,T_n)$.
Each $q_i$ is either a vertex or an edge, and then $\lk(q_i,T_i)$ is
either a finite set or the empty set, accordingly.
Let $C$ be set of those $i$ for which $q_i$ is a vertex.
Then the cube $\fll q\flr$ is of dimension $n-\#C$, and consequently
the dimension $d-1$ of $\lk(q,B)$ equals $k-n+\#C-1<\#C-1$.

Every vertex $v$ of $\lk(q,T)$ lies in $\lk(q_i,T_i)$ for some $i\in C$;
in that case let us color $v$ by the $i$th color.
In particular, the subcomplexes $\Lambda:=\lk(q,Q)$ and $S:=\lk(q,B)$ of $\lk(q,T)$
are $C$-colored.
Since $\#C>d$, by the Fisk--Izmestiev--Witte lemma, the $C$-colored
combinatorial $(d-1)$-sphere $S$ bounds (abstractly) a $C$-colored combinatorial
ball $D$.
Let $\Lambda^+$ be the amalgamated union $\Lambda\cup_S D$, that is, the pushout of
the diagram $\Lambda\supset S\subset D$ in the category of $C$-colored
simplicial complexes and color-preserving simplicial maps.

If $D\but S$ contains $k_i$ vertices of color $i$, we define a new tree
$T_i^+=T_i\cup (q_i*[k_i])$ by attaching $k_i$ new edges to $T_i$ at the vertex $q_i$
for each $i\in C$ (note that $[k_i]=\emptyset$ and so $T_i^+=T_i$ for each $i\notin C$).
Let $T^+=T_1^+\x\dots\x T_n^+$.
The $C$-coloring of the vertices of $\lk(q,T)$ extends to the similarly defined
$C$-coloring of the vertices of $\lk(q,T^+)$.
Then any color-preserving identification of the vertices of $D\but S$ with
the vertices of $\lk(q,T^+)$ that are not in $\lk(q,T)$ extends uniquely to
a color-preserving simplicial map $\Lambda^+\to\lk(q,T^+)$ that extends
the inclusion $\Lambda\subset\lk(q,T)$.
This simplicial map is injective on vertices, hence is an embedding.
By construction, $\Lambda^+\cap\lk(q,T)=\Lambda$.

In particular, $D$ is now identified with a subcomplex of $\lk(q,T)$, hence
$E:=\fll q\flr\x(C^*D)^\#$ and $F:=\fll q\flr\x D^\#\cup(\partial\fll q\flr)\x (C^*D)^\#$
are identified with subcomplexes of $\fll q\flr\x (C^*\lk(q,T^+))^\#=\st(q,T^+)$.
Since $D\cap\Lambda=S$, we have $E\cap Q=E\cap B$.
Let $Q^+=Q\cup E$.
Note that $E\cap B$ is the cubical combinatorial $k$-ball
$\st(q,B)=\fll q\flr\x(C^*\lk(q,B))^\#$, and $F\cap B$ is its boundary, the cubical
combinatorial sphere
$\partial\st(q,B)=\fll q\flr\x\lk(q,B)^\#\cup(\partial\fll q\flr)\x (C^*\lk(q,B))^\#$.
Further note that $\st(q,B)\but\partial\st(q,B)$ is the dual cone
$\cel q\cer$ of $q$ in $B$.
Then $B^+=(B\but\cel q\cer)\cup F$ is a cubical combinatorial $k$-ball,
which does not contain $q$.

Since $\Lambda^+\cap\lk(q,T)=\Lambda$, we have $Q^+\cap T=Q$.
Furthermore, $|T^+|$ collapses onto $|T\cup (q_1*[k_1])\x\dots\x (q_n*[k_n])|$
(using conewise collapses of the form $X\x CY\searrow X\cup Z\x CY$ where $Z$ is
a closed subpolyhedron of $X$), which in turn collapses onto $|T\cup E|=|T\cup Q^+|$
(using the collapse of the cone $|\prod_{i\in C}q_i*[k_i]|$ onto its subcone
$|(C^*D)^\#|$).

In order to fit the above process in an inductive argument, let us now write
$Q_0$, $B_0$ for the given $Q$, $B$.
Assuming that $Q_i$, $B_i$ have been constructed, along with some distinct
$q_1,\dots,q_i\in (B_0\but\partial B_0)\but B_i$, we repeat the above process
with $Q=Q_i$ and $B=B_i$, with one modification: $q$ is now not an arbitrary
face of $B_i\but\partial B_i$, but one that is also a face of the original
$B_0\but\partial B_0$.
Since $q$ is still required to be a face of $B_i$, our hypothesis entails that
$q\notin\{q_1,\dots,q_i\}$.
We set $Q_{i+1}=Q^+$, $B_{i+1}=B^+$, and $q_{i+1}=q$.
Then $q_0,\dots,q_{i+1}\in (B_0\but\partial B_0)\but B_{i+1}$, which completes
the inductive step.
Since $B_0\but\partial B_0$ is finite, the number of steps is bounded.
If the final step is $r$th, it is easy to see that $B_r\cap
B_0=\partial B_0=\partial B_r$, and $B_0\cup B_r$ bounds a cubical combinatorial
$(k+1)$-ball $\beta$ (namely, $\beta$ is the union of all the $(k+1)$-balls of
the form $E$) such that $\beta\cap Q_0=B_0$ and $\beta\cup Q_0=Q_r$.
\end{proof}

\begin{remark} The combinatorial type of the ball $\beta$ depends on the
order in which $q_1,\dots,q_r$ are picked out of $B_0\but\partial B_0$.
For instance, suppose that $n=2$, $k=1$ and the arc $B_0$ consists of $e$ edges
(and hence $e+1$ vertex).
If $e>1$, then we may take $q_1,\dots,q_r$ to be all the non-boundary vertices,
ordered consecutively, which will lead to the same $\beta$ as in \cite{KKS}.
For instance if $e=2$ (so $r=1$) and $T_1=Q_0=B_0$, $T_2=pt$, then $\tilde T_1=T_1$,
$\tilde T_2$ is a single edge, and $Q_r=B_r=\tilde T_1\x\tilde T_2$ (which
amounts to two squares).
On the other hand, if we first pick out all the edges (in any order) and then
the $e-1$ non-boundary vertices (in any order), the result will be unique, but
quite different from the above.
For instance if $e=2$ (so $r=3$) and $T_1=Q_0=B_0$, $T_2=pt$, then at the
final step $\tilde T_1$ is a triod, $\tilde T_2$ contains two edges, and $B_r$
consists of four squares.
Picking out only vertices but not consecutively may also lead to a $\beta$
different from that in \cite{KKS}.
\end{remark}

\begin{remark} As discussed in the previous remark, the construction in the proof
of Lemma \ref{main lemma} depends on the choices of the cubes $q_1,\dots,q_r$.
Let us describe a canonical range of choices that all lead to the same embedding.
Each tree $T_i$ is constructed in stages $pt=T_{i0}\subset\dots\subset T_{is}=T_i$.
The vertices of $T_i$ are partially ordered by $v<w$ if there exists a $k<s$
such that $v\in T_{ik}$ and $w\notin T_{ik}$, yet $w$ and $v$ belong to the same
component of $|T_i\but\cel T_{i,k-1}\cer|$.
(In particular, incomparable vertices are non-adjacent in the tree.)
This yields a partial order on the vertices of
$B\but\partial B\subset Q\subset T_1\x\dots\x T_n$.
Let $q_1,\dots,q_r$ be the vertices of $B\but\partial B$ arranged in some
total order extending the constructed partial order.
It is clear then that $r$ is indeed the last stage of the construction, and that
$Q_r$ does not depend on the choice of the total order.
\end{remark}

An alternative proof of the implication (iii)\imp(i) in Theorem \ref{characterization} is
given by Lemma \ref{main lemma} along with the following lemma (take $k=n-1$).

\begin{lemma}\label{standard embedding}
Let $L$ be an $k$-dimensional simplicial complex.
Then there exist trees $T_0,\dots,T_k$ and a PL embedding $f\:|L|\to |T|$,
where $T=T_0\x\dots\x T_k$ such that $f(|\sigma|)$ is cubulated by a
subcomplex of $T$ for every simplex $\sigma$ of $L$.
\end{lemma}

The {\it prejoin} $P+Q$ consists of the elements of $P\cup Q$ with
the order $\preceq$ defined as follows: $p\preceq q$ iff either
$p,q\in P$ and $p\le q$ in $P$; or $p,q\in Q$ and $p\le q$ in $Q$;
or $p\in P$ and $q\in Q$.
Note that $C^*P\simeq pt+P$.
It is easy to see that $(P+Q)^\flat\simeq P^\flat+Q^\flat$, where
$P^\flat$ denotes the barycentric subdivision (see details in \cite{M3}).

\begin{proof}
Let $S_i$ be the set of $i$-dimensional simplices of $L$.
Then $L$ is a subcomplex of $S_0+\dots+S_k$.
Hence $L^\flat$ is a subcomplex of $(S_0+\dots+S_k)^\flat\simeq S_0*\dots*S_k$,
which in turn is a subcomplex of $C^*(S_0*\dots*S_k)$.
Therefore $(L^\flat)^\#$ is a subcomplex of $(C^*(S_0*\dots*S_k))^\#\simeq
(C^*S_0)^\#\x\dots\x(C^*S_k)^\#$.
Each $(C^*S_i)^\#$ is a tree, and the assertion follows.
\end{proof}

\section{Local cohomology}

By $H^*$ we denote the Alexander--Spanier cohomology \cite{Sp}, \cite{Ma},
or equivalently (see \cite{Sk2}) sheaf cohomology with constant coefficients
\cite{Br}.
If the coefficients are omitted, they are understood to be integer.
The case of coefficients in a field is much easier (see \cite{Wh}) but
will not suffice for our purposes.

If $(X,Y)$ is a pair of compacta, $H^i(X,Y)$ is isomorphic to the direct limit
$\dirlim H^i(P_j,Q_j)$, where $\dots\to(P_1,Q_1)\to(P_0,Q_0)$ is any
inverse sequence of pairs of compact polyhedra with inverse limit $(X,Y)$.
In particular, every cohomology group $H^i(Y,X)$ is countable.

More generally, when $Y$ is closed in $X$ (which we always assume to be
metrizable), then $H^i(X,Y)$ coincides (see \cite{Sk2}) with the \v Cech
cohomology of $(X,Y)$, which may be defined as the direct limit of the
$i$th cohomology groups of the nerves of all open coverings of $(X,Y)$.
In particular, if $Y$ is closed in $X$ and $X$ is $n$-dimensional, then
$H^i(X,Y)=0$ for $i>n$ (since covers with at most $n$-dimensional nerve form
a cofinal subset in the directed set of all open covers of $X$).

If $X$ is a compactum and $x\in X$, the local cohomology group
$H^i(X,\,X\but\{x\})$ is isomorphic to $\dirlim H^{i-1}(U_i\but\{x\})$,
where $U_0\supset U_1\supset\dots$ are neighborhoods of $x$ in $X$ such that
$\bigcap U_k=\{x\}$ and each $\Int U_k\supset\Cl U_{k+1}$.
As observed in \cite[\S1]{Sk1}, this follows from the exact sequences of
the pairs $(U_k,U_k\but\{x\})$ and
the fact that the direct limit functor preserves exactness of sequences.
However, this isomorphism will not be used in the sequel.

Instead, we shall use the following more geometric description of the local
cohomology groups (parallel to \cite[proof of Lemma 1]{Mi2}).

\begin{proposition}\label{prop:u0infty}
Let $X$ be a compactum, let $x\in X$ and let $U_1\supset U_2\supset\dots$
be neighborhoods of $x$ in $X$ such that $\bigcap U_k=\{x\}$ and each
$\Int U_k\supset\Cl U_{k+1}$.
Then
$$H^i(X,\,X\but\{x\})\simeq
H^i(X\x [0,\infty),\,X\x [0,\infty)\but U_{[0,\infty)}),$$
where
$U_{[0,\infty)}=U_0\x [0,1)\ \cup\ U_1\x [1,2)\ \cup\ U_2\x [2,3)\ \cup\ \dots$.
\end{proposition}

Note that if the $U_k$ are open, then $X\x [0,\infty)\but U_{[0,\infty)}$ is
a closed subset of $X\x [0,\infty)$.
Hence from the preceding discussion we obtain

\begin{corollary}\label{cor:nplus1}
If $X$ is an $n$-dimensional compactum, $H^i(X,\,X\but\{x\})=0$
for $i>n+1$ and all $x\in X$.
\end{corollary}

\begin{proof}[Proof of Proposition \ref{prop:u0infty}]
We shall show that $(X,\,X\but\{x\})$ is ``almost'' homotopy equivalent to
the mapping telescope of pairs $(X,\,X\but U_i)$, meaning that there is a map
of pairs in one direction, which admits a homotopy inverse separately on
each entry of the pair; by the Five Lemma, this is just good enough as long
as cohomology is concerned.

The projection $X\x [0,\infty)\to X$ yields a map of pairs
$f\:(X\x[0,\infty),\,X\x[0,\infty)\but U_{[0,\infty)})\to
(X,\,X\but\{x\})$.
If $\phi\:X\but\{x\}\to[0,\infty)$ is a map such that
$\phi^{-1}([0,n])\subset X\but U_n$, then
$g\:X\but\{x\}\to X\x[0,\infty)$ defined by $g(y)=(y,\phi(y))$ is an embedding
into $X\x [0,\infty)\but U_{[0,\infty)}$.
It is easy to see that $g$ is homotopy inverse to the restriction
$h\:X\x[0,\infty)\but U_{[0,\infty)}\to X\but\{x\}$ of the projection
$X\x[0,\infty)\to X$; hence $h$ is a homotopy equivalence.
Using the isomorphisms induced by $g$ and the homotopy
equivalence $X\x [0,\infty)\to X$, the Five Lemma implies that
$f^*$ is an isomorphism.
\end{proof}

By well-known arguments (see \cite[proof of Theorem 4]{Mi1} or
\cite[proof of equation ($*$) in \S1.B or proof of Theorem 3.1(b)]{Me}),
Proposition \ref{prop:u0infty} gives rise to a Milnor-type natural short exact
sequence (found explicitly in \cite{Ha}):
$$0\to\derlim H^{i-1}(X,\,X\but U_k)\to H^i(X,\,X\but\{x\})\to
\invlim H^i(X,\,X\but U_k)\to 0.$$
In particular,
\begin{equation}
H^{n+1}(X,\,X\but\{x\})\simeq\derlim H^n(X,\,X\but U_k),\tag{$*$}
\end{equation}
if $X$ is an $n$-dimensional compactum.

\begin{lemma}\label{lem:product} If $X$ and $Y$ are compacta of dimensions
$n$ and $m$, and $x\in X$ and $y\in Y$ are such that $H^{n+1}(X,\,X\but\{x\})=0$
and $H^{m+1}(Y,\,Y\but\{y\})=0$, then also
$H^{n+m+1}(X\x Y,\,X\x Y\but\{(x,y)\})=0$.
\end{lemma}

\begin{proof} Since cohomology groups of pairs of compacta are countable,
the hypothesis and the conclusion can be reformulated in terms of
the Mittag-Leffler condition, using the isomorphism ($*$) and Gray's Lemma
(see \cite[Lemma 3.3]{Me}).
Then the assertion follows (cf.\ \cite[proof of Lemma 3.6(b)]{MS}) from
the naturality in the K\"unneth formula \cite[Theorem II.15.2 and
Proposition II.12.3]{Br} (see also \cite[Theorem 7.1]{Ma}, which implies
the relative case using the map excision axiom).
\end{proof}

\begin{lemma}\label{lem:incl}
If $X$ is an $n$-dimensional compactum and
$H^{n+1}(X,\,X\but\{x\})=0$, then $H^{n+1}(Y,\,Y\but\{x\})=0$ for every
$n$-dimensional compactum $Y\i X$ containing $x$.
\end{lemma}

\begin{proof} Let $U_k$ be open neighborhoods of $x$ in $X$ as in Proposition
\ref{prop:u0infty}.
The restriction map
$H^n(X,\,X\but U_k)\xr{f_k}H^n(Y,\,Y\but U_k)$ is onto from the exact
sequence of the triple $(X,\,Y\cup (X\but U_k),\,X\but U_k)$, due to
$H^{n+1}(X,\,Y\cup (X\but U_k))=0$.
Then $\derlim f_k$ is onto from the six-term exact sequence of inverse and
derived limits (see \cite[Theorem 3.1(d)]{Me} for a geometric proof) associated
to the short exact sequences
$$0\to\ker f_k\to H^n(X,\,X\but U_k)\xr{f_k}H^n(Y,\,Y\but U_k)\to 0.$$
But by naturality of the isomorphism ($*$), $\derlim f_k$ is identified with
the restriction map $H^{n+1}(X,\,X\but\{x\})\to H^{n+1}(Y,\,Y\but\{y\})$.
\end{proof}

\begin{remark}\label{rem:menger}
The Menger curve $M$ contains points $x$ such that
$H^2(M,\,M\but\{x\})\ne 0$.
(Since $M$ is known to be homogeneous, this applies to every $x\in M$.)
For let $Y$ be the subspace $\N^+\x [0,1)\cup [0,\infty]\x\{1\}$ of
$[0,\infty]\x [0,1]$, where $\N^+=\{0,1,\dots,\infty\}$, and
let $y=(\infty,1)\in Y$.
Let us represent $Y\but\{y\}$ as a union $\bigcup K_i$, where each $K_i$ is
compact and lies in $\Int K_{i+1}$. (And not just in $K_{i+1}$.)
Then $\dots\to\tilde H^0(K_1)\to\tilde H^0(K_0)$ is of the form
$\dots\to\bigoplus_{S_1}\Z\to\bigoplus_{S_0}\Z$, where
$S_0\supset S_1\supset\dots$ is a nested sequence of infinite countable sets
with $\bigcap S_i=\emptyset$.
Since $H^1(Y)=0=\tilde H^0(Y)$, the inverse sequence
$\dots\to H^1(Y,K_1)\to H^1(Y,K_0)$ is of the same form.
Clearly it does not satisfy the Mittag-Leffler condition and consists of
countable groups, so by Gray's Lemma (see \cite[Lemma 3.3]{Me})
its derived limit is nontrivial.
(In fact, it is easy to see, similarly to \cite[Example 3.2]{Me}, that
$\derlim H^1(Y,K_i)\simeq\prod_\N\Z/\bigoplus_\N\Z$.)
Thus by ($*$), $H^2(Y,\,Y\but\{y\})\ne 0$.
Since $Y$ embeds into $M$, Lemma \ref{lem:incl} implies
$H^2(M,\,M\but\{x\})\ne 0$, where $x$ is the image of $y$.
\end{remark}

\begin{lemma}\label{lem:1-dim} If $X$ is a local dendrite, then
$H^2(X,\,X\but\{x\})=0$ for every $x\in X$.
\end{lemma}

The proof is a bit technical; let us explain informally some intuition behind it.
There are just two basic examples of inverse sequences of countable abelian groups
with nonzero $\derlim$: (i) $\dots\xr{p_1}\Z\xr{p_0}\Z$, each $p_i$ being
a nonzero prime (this occurs in the Skliarienko compactum), and (ii)
$\dots\emb\bigoplus_{i=1}^\infty\Z\emb\bigoplus_{i=0}^\infty\Z$
(this occurs in Remark \ref{rem:menger} and is called ``Jacob's ladder'' in
\cite{HR}).
Example (i) cannot occur in ($*$) with $n=1$, because there is ``not enough room
for twisting'' in one-dimensional spaces, so we cannot expect to find even a single
multiplication as in (i).
On the other hand, if $X$ is an LC$_n$ compactum, then we cannot find example (ii)
in ($*$), because $n$-cohomology of compact subsets of $X$ is ``almost'' finitely
generated in the sense that for every two compact subsets $K\i X$ and $L\i\Int K$,
the image of $H^n(K)\to H^n(L)$ is finitely generated \cite[II.17.5 and V.12.8]{Br}.

\begin{proof} Let us represent $X\but\{x\}$ as a union $\bigcup K_i$, where
each $K_i$ is compact and lies in $\Int K_{i+1}$.
Since $X$ is locally contractible, for each $n$ (in particular, for $n=1$), each
inclusion map $K_i\to K_{i+1}$ factors through a (not necessarily embedded in $X$)
LC$_n$ compactum $L_i$ \cite[Theorem 6.11]{Me}.
We recall that LC$_n$ compacta have finitely generated cohomology and (Steenrod)
homology in dimensions $\le n$ (see \cite[II.17.7 and V.12.8]{Br},
\cite[6.11]{Me}).
Universal coefficients formulas then imply that LC$_1$ compacta have free abelian
$H^1$ (see \cite[V.12.8]{Br}) and consequently also free abelian $H_0$ (see
\cite[\S V.3, Eq.\ (9) on p.\ 292]{Br}).

Consider a composition $f\:L_i\to K_{i+1}\to K_j\to L_j$.
By the naturality of the universal coefficients formula (see
\cite[V.12.8, V.13.7]{Br}), $f^*\:H^0(L_j)\to H^0(L_i)$ is dual to
$f_*\:H_0(L_i)\to H_0(L_j)$.
The image of $f_*$ is a subgroup of the free abelian group $H_0(L_j)$.
So it is itself free abelian, in particular, projective as a $\Z$-module.
Hence $f_*$ is a split epimorphism onto its image.
Then the inclusion of the image of $f^*$ into $H^0(L_i)$ is a split
monomorphism.
(Indeed, given abelian group homomorphisms
$f_*\:G\to H$, $f^*\:\Hom(H,\Z)\to\Hom(G,\Z)$ defined by $f^*(\psi)=\psi f_*$, and
$s\:\im f_*\to G$ such that $f_*sf_*=f_*$, define $r\:\Hom(G,\Z)\to\im f^*$ by
$r(\phi)=\phi sf_*$; then $rf^*=f^*$, i.e.\ $r(\psi f_*)=\psi f_*$ for each
$\psi\in\Hom(H,\Z)$.)
Thus $f^*$ is a homomorphism onto a direct summand of $H^0(L_i)$.
The finitely generated group $H^0(L_i)$ contains no infinitely decreasing
chain of direct summands; so the inverse sequence
$\dots\to H^0(L_1)\to H^0(L_0)$ satisfies the Mittag-Leffler condition.
Hence so does $\dots\to H^0(K_1)\to H^0(K_0)$.

On the other hand, consider a composition $g\:L_i\to K_{i+1}\to X$.
The image of $g^*\:H^1(X)\to H^1(L_i)$ is a subgroup of the free abelian
group $H^1(L_i)$.
So it is itself free abelian, in particular, projective as a $\Z$-module.
Hence $g^*$ is a split epimorphism onto its image.
Then the kernel of $g^*$ is a direct summand in $H^1(X)$.
The finitely generated group $H^1(X)$ contains no infinitely decreasing
chain of direct summands; hence the homomorphisms $H^1(X)\to H^1(L_i)$
have the same kernel for all sufficiently large $i$.
Then so do the homomorphisms $H^1(X)\to H^1(K_i)$.
Since $X$ is $1$-dimensional, the latter are surjective.
Hence $H^1(K_{i+1})\to H^1(K_i)$ are isomorphisms for sufficiently
large $i$.
In particular, $\dots\to H^1(K_1)\to H^1(K_0)$ satisfies the dual
Mittag-Leffler condition.

Thus by Dydak's Lemma (see \cite[Lemma 3.11]{Me}),
$\dots\to H^1(X,K_1)\to H^1(X,K_0)$ satisfies the Mittag-Leffler
condition.
Hence $\derlim H^1(X,K_i)=0$, and the assertion follows from ($*$).
\end{proof}

\section{Skliarienko's compactum}

We note that if the compactum $X$ is the limit of an inverse sequence of compacta
$X_i$, all of which embed in $Y$, then $X$ quasi-embeds in $Y$ (for it follows
from the definition of the topology of the inverse limit that the maps
$X\xr{p^\infty_i}X_i\i Y$ are $\eps_i$-maps with respect to any fixed metric on $X$,
where $\eps_i\to 0$ as $i\to\infty$).
The converse implication (which we shall not need here) holds when $Y$ is
a polyhedron (a simple proof should appear in a future version of \cite{MS};
see also \cite[Theorem 1]{MaS} but beware that their ``$\eps$-maps'' are
required to be surjective).

\begin{definition}[Skliarienko's compactum]
Given a direct sequence $X_1\to X_2\to \dots$, the {\it mapping telescope}
$\Tel(X_1\to X_2\to\dots)$ is the infinite union
$MC(X_1\to X_2)\cup_{X_2} MC(X_2\to X_3)\cup_{X_3}\dots$
of the mapping cylinders (the direct limit of the finite unions).
Let $X$ be the one-point compactification of
the mapping telescope of the direct sequence
$$S^1\xr{2}S^1\xr{2}\ldots$$ of two-fold coverings.
It is easy to see that $X$ is a contractible and locally contractible
$2$-dimensional compactum, and so an AR.
It was introduced by Je.\ G. Skliarienko \cite[Example 4.6]{Sk1}.
We shall call $X$ the {\it Skliarienko compactum}.
\end{definition}

\begin{proposition} \label{thm:quasi-emb} Skliarienko's compactum quasi-embeds in
a product of two dendrites.
\end{proposition}

\begin{proof} Let us represent $X$ as an inverse limit of polyhedra.
To this end, consider the following mapping telescope of a direct sequence:
$$X_i=\Tel(S^1_1\xr{2}\ldots\xr{2}S^1_i\to pt),$$
where each $S^1_j$ stands for a copy of $S^1$.
Note that $X$ contains the cone $D^2=\Tel(S^1_i\to pt)$.
Let $f_i\colon X_{i+1}\to X_i$ be the composition of the quotient map
$X_{i+1}\to X_{i+1}/D^2$ and a homeomorphism $X_{i+1}/D^2\to X_i$
which is the identity on $\Tel(S^1_1\xr{2}\ldots\xr{2}S^1_i)$.
Then $X$ is homeomorphic to the inverse limit of
$\dots\xr{f_2}X_2\xr{f_1}X_1$.

Notice that each $X_i$ is a collapsible 2-polyhedron.
Hence by a result of Koyama, Krasinkiewicz and Spie\.z (see Corollary
\ref{2d-case}), $X_i$ embeds in a product of two trees $T_i$ and $T'_i$.
Let us consider the cluster
$T=\invlim(\dots\to T_1\vee T_2\vee T_3\to T_1\vee T_2\to T_1)$ of
the $T_i$, where the basepoint of each $T_i$ is one of its endpoints.
Let $T'$ be the analogous cluster of the trees $T'_i$.
Then $T$ and $T'$ are dendrites, $T$ contains a copy of each $T_i$,
and $T'$ contains a copy of each $T'_i$.
Thus each $X_i$ embeds in $T\x T'$.
Therefore $X$ quasi-embeds there.
\end{proof}

Let $X$ be the Skliarienko compactum and let $\infty\in X$ denote the
remainder point of the one-point compactification.
It is easy to see that $H^3(X,\,X\but\{\infty\})$ is non-zero \cite{Sk1}.
More generally, let us compute $H^{3+k}(X\x I^k,\,X\x I^k\but\{(\infty,0)\})$,
where $I=[-1,1]$.
Let $F_i$ be the union of the first $i$ mapping cylinders in the mapping telescope:
$$F_i=\Tel(S^1_1\xr{2}\ldots\xr{2} S^1_i).$$
Each $F_i$ collapses onto $S^1_i$, and these collapses identify up to homotopy
the inclusions $F_i\i F_{i+1}$ with the two-fold coverings $S^1_i\xr{2}S^1_{i+1}$.
Hence the inverse sequence $\dots\to H^1(F_2)\to H^1(F_1)$ is of the form
$\dots\xr2\Z\xr2\Z$.
Since $X$ is an AR, so is the inverse sequence $\dots\to H^2(X,F_2)\to H^2(X,F_1)$.
Let $G_i=F_i\x I^k\cup X\x (I^k\but (-\frac1i,\frac1i)^k)$.
By the K\"unneth formula (see references in the proof of Lemma \ref{lem:product}),
$H^{2+k}(X\x I^k,G_i)\simeq H^2(X,F_i)$, and the inverse sequence
$\dots\to H^{2+k}(X\x I^k,G_2)\to H^{2+k}(X\x I^k,G_1)$ is again of the form
$\dots\xr2\Z\xr2\Z$.
In particular, it does not satisfy the Mittag-Leffler condition, so by
Gray's Lemma (see \cite[Lemma 3.3]{Me}) its derived limit is nontrivial.
(In fact, it is easy to compute that it is isomorphic to $\Z_2/\Z$, where $\Z_2$
is the group of $2$-adic integers; see \cite[Example 3.2]{Me}.)
Thus by ($*$), $H^{3+k}(X\x I^k,\,X\x I^k\but\{(\infty,0)\})\ne 0$.

\begin{theorem}\label{thm:non-emb}
If $X$ is the Skliarienko's compactum, $X\x I^k$ does not embed in any product of
$2+k$ local dendrites.
\end{theorem}

\begin{proof} Suppose $X\x I^k\i Y_1\x\dots\x Y_n$, where $Y_i$ are local dendrites.
Then $(\infty,0)\in X\x I^k$ is of the form $(y_1,\dots,y_n)$.
By Lemma \ref{lem:1-dim}, $H^2(Y_i,\,Y_i\but\{y_i\})=0$ for each $i$.
Then by Lemma \ref{lem:product}, $H^{3+k}(\prod Y_i,\,\prod Y_i\but\{(y_i)\})=0$.
Therefore by Lemma \ref{lem:incl}, $H^{3+k}(X\x I^k,\,X\x I^k\but\{(\infty,0)\})=0$.
This contradicts the above computation.
\end{proof}

\begin{theorem}[Koyama--Krasinkiewicz--Spie\.z] \label{thm:kras} If a compact $n$-dimensional ANR
embeds in a product of $n$ curves, then it embeds in a product of $n$ local
dendrites.
\end{theorem}

\begin{proof} It is well-known that locally contractible compacta have finitely
generated cohomology groups (see \cite[II.17.7]{Br}, \cite[6.11]{Me}).
If a locally connected $n$-dimensional compactum $X$ with finitely generated
$H^n(X)$ embeds in a product $n$ curves, then the first several lines of the proof
of Theorem 2.B.1 in \cite{KKS} (which contain further references) produce
an embedding of $X$ in a product of $n$ local dendrites.
\end{proof}

Theorems \ref{thm:non-emb} and \ref{thm:kras} have the following

\begin{corollary} \label{cor} Skliarienko's compactum multiplied by $I^k$ does not
embed in any product of $2+k$ curves.
\end{corollary}

Corollary \ref{cor} combines with Proposition \ref{thm:quasi-emb} to imply
Theorem \ref{thm:main}.

\begin{remark} If $\dots\to G_1\to G_0$ is an inverse sequence of countable groups,
let $\derlim_{fg} G_i$ be the direct limit $\dirlim L_\alpha$ of the derived limits
$L_\alpha=\derlim H_{\alpha i}$ over all inverse sequences $\dots\to H_{\alpha
1}\to H_{\alpha 0}$ of finitely generated subgroups $H_{\alpha_i}\subset G_i$,
where the bonding maps are the restrictions of those in $\dots\to G_1\to G_0$.
Some results about $\derlim_{fg}$ will appear in a future paper by the first author.
By using the functor $\derlim_{fg}$ in place of $\derlim$, it should be possible
to refine the proof of Theorem \ref{thm:non-emb} so as to obtain a purely algebraic
proof of Corollary \ref{cor}, without using Theorem \ref{thm:kras}.
\end{remark}

\begin{remark} The same arguments (only using the general case of Theorem
\ref{collapsible} rather than the easier $2$-dimensional case) show that
the $n$-dimensional Skliarienko compactum (similarly defined with $S^{n-1}$
in place of $S^1$) quasi-embeds in a product of $n$ dendrites, but does not
embed in a product of $n$ curves.
\end{remark}

\section{Co-local contractibility}\label{co-local}

Let us call a compactum $X$ {\it co-locally contractible} at $x\in X$ if
every neighborhood $U$ of $x$ contains a neighborhood $V$ of $x$ such that
the inclusion $X\but \{x\}\i X$ is homotopic to a map $X\but\{x\}\to X\but V\i X$
by a homotopy keeping $X\but U$ fixed.
(Equivalently, every neighborhood $U$ of $x$ contains a neighborhood $V$
of $x$ such that for every neighborhood $W$ of $x$ contained in $V$,
the inclusion $X\but W\i X$ is homotopic to a map $X\but W\to X\but V$
by a homotopy keeping $X\but U$ fixed.)
We call $X$ {\it co-locally contractible} if it is co-locally contractible
at every point.
(Compare Borsuk's idea of colocalization \cite[\S IX.16]{Bo1} and colocal
connectedness of Krasinkiewicz and Minc \cite{KM}.)

\begin{remark}
A slightly stronger property than co-local contractibility, obtained by
replacing the inclusion $X\but \{x\}\i X$ with the identity map of $X\but \{x\}$,
is known as {\it reverse} (or {\it backward}) {\it tameness} of $X\but \{x\}$
(see \cite{Qu}, \cite{HR}).
Dually, $X\but \{x\}$ is called {\it forward tame} if there exists a closed
neighborhood $U$ of $x$ such that for every neighborhood $V$ of $x$,
the inclusion $V\but \{x\}\i X\but \{x\}$ is properly homotopic to a map
$V\but \{x\}\to U\but \{x\}\i X\but \{x\}$ (see \cite{Qu}, \cite{HR}).
It is not hard to see (even if appears surprising) that forward
tameness of $X\but\{x\}$ implies local contractibility of $X$ at $x$.
To see that the converse implication fails, let $P$ be the suspension of
a non-contractible acyclic polyhedron and let its basepoint $b$ be one of
the two suspension points; or alternatively let $P$ be the dunce hat and
$b$ its unique $0$-cell.
Then the cluster $C=\invlim(\dots\to P\vee P\vee P\to P\vee P\to P)$ of
copies of $P$ is an AR, yet it follows from Dydak--Segal--Spie\. z \cite{DSS}
that $C\but\{b\}$ is not forward tame.
\end{remark}

\begin{proposition} If an $n$-dimensional compactum $X$ is co-locally
contractible at $x$, then $H^{n+1}(X,\,X\but\{x\})=0$.
\end{proposition}

\begin{proof} This is a straightforward diagram chasing.
The hypothesis implies that, with $x$, $U$ and $V$ as above and for each $i$,
the restriction map $H^i(X\but \{x\})\to H^i(X\but V)$ is a split injection
on the image of $H^i(X)$.
Hence the image of the forgetful map
$f\:H^i(X\but\{x\},\,X\but V)\to H^i(X\but \{x\})$ lies in the image of $H^i(X)$.
The latter equals the kernel of the coboundary map
$\delta\:H^i(X\but x)\to H^{i+1}(X,\,X\but\{x\})$, hence $\delta f=0$.
Since this $\delta f\:H^i(X\but \{x\},\,X\but V)\to H^{i+1}(X,\,X\but \{x\})$ is
also the coboundary map, the restriction
$H^{i+1}(X,\,X\but \{x\})\to H^{i+1}(X,\,X\but V)$ must be an injection.
Finally, since $X$ is $n$-dimensional and without loss of generality $V$ is
open, $H^{n+1}(X,\,X\but V)=0$.
Thus $H^{n+1}(X,\,X\but \{x\})=0$.
\end{proof}

{\bf Acknowledgments.}
The authors are grateful to Professors Krasinkiewicz and Spie\.z for posing
the problem and for stimulating discussions, and the referee for helpful
remarks.
The first author also thanks J. Dydak, O. Frolkina and Je.\ V. Shchepin for
relevant conversations.

\end{document}